\documentclass[reqno,11pt, twoside]{amsart}
\usepackage{color,amssymb,latexsym,amsfonts,textcomp,lastpage,fancyhdr,calc,graphicx}
\usepackage{amsmath,amstext,amsthm,amssymb,amsxtra}
\usepackage{txfonts}

\usepackage[b5paper, top=2.8cm, bottom=2.8cm, lmargin=2.4cm, rmargin=2.8cm]{geometry}
\usepackage{microtype}
\usepackage[T1]{fontenc}
\usepackage{nohyperref}

\newtheorem{theorem}{Theorem}
\newtheorem{corollary}[theorem]{Corollary}
\newtheorem{lemma}{Lemma}

\newtheorem{conjecture}{Conjecture}

\newtheorem{observation}{Observation}
\newtheorem{definition}{Definition}
\theoremstyle{definition}

\newcommand{\beql}[1]{\begin{equation}\label{#1}}
\newcommand{\eeq}{\end{equation}}
\newcommand{\comment}[1]{}

\newcommand{\Abs}[1]{{\left|{#1}\right|}}

\newcommand{\Set}[1]{{\left\{{#1}\right\}}}

\newcommand{\RR}{{\mathbb R}}

\newcommand{\CC}{{\mathbb C}}
\newcommand{\ZZ}{{\mathbb Z}}

\newcommand{\one}{{\bf 1}}

\newcommand{\gen}[1]{{\langle #1 \rangle}}

\newcommand{\modu}{{\rm mod\,}}

\newcommand{\ft}[1]{\widehat{#1}}

\newcounter{rem}
\newcounter{step}
\newcounter{mysec}
\newcounter{mysubsec}[mysec]
\newcounter{othm}
\def\theothm{\Alph{othm}}
\newenvironment{othm}{
  \em
  \vskip 0.10in
  \refstepcounter{othm}
  \noindent{\bf Theorem\ \theothm}
}

\definecolor{ble}{rgb}{0,0.33,0.75}

\begin{document}

\title{Spectra for cubes in products of finite cyclic groups}

\author[E. Agora]{{Elona Agora}}
\author[S. Grepstad]{{Sigrid Grepstad}}
\author[M. Kolountzakis]{{Mihail N. Kolountzakis}}
\address{E.A.: Instituto Argentino de Matem\' atica ``Alberto P. Calder\' on'' (IAM-CONICET), Argentina}
\email{elona.agora@gmail.com}
\address{M.K.: Department of Mathematics and Applied Mathematics, University of Crete, Voutes Campus, GR-700 13, Heraklion, Crete, Greece}
\email{kolount@gmail.com}
\address{S.G.: Institute of Financial Mathematics and Applied Number Theory, Johannes Kepler University Linz, Austria.}
\email{sgrepstad@gmail.com}

\thanks{This work has been partially supported by the ``Aristeia II'' action (Project
FOURIERDIG) of the operational program Education and Lifelong Learning
and is co-funded by the European Social Fund and Greek national resources.}
\thanks{E.A.\ has been partially supported by Grants:  MTM2013-40985-P, 2014SGR289, 
CONICET-PIP 11220110101018, UBACyT 2002013010042BA}

\thanks{S.G.\ is currently supported by the Austrian Science Fund (FWF), Project F5505-N26, which is a part of the Special Research Program ``Quasi-Monte Carlo Methods: Theory and Applications''}

\begin{abstract}
We consider ``cubes'' in products of finite cyclic groups and we study their tiling and spectral properties.
(A set in a finite group is called a tile if some of its translates form a partition of
the group and is called spectral if it admits an orhogonal basis of characters for the functions supported
on the set.)
We show an analog of a theorem due to Iosevich and Pedersen \cite{IP}, Lagarias, Reeds and Wang \cite{LRW},
and the third author of this paper \cite{K00}, which identified the tiling complements of the unit cube in $\RR^d$ with
the spectra of the same cube.
\end{abstract}

\maketitle

\section{Introduction to tilings and spectra}

Let $G$ be a locally compact abelian group equipped with Haar measure,
which is always taken to be the counting measure on discrete groups.
(We will deal exclusively with finite groups in this paper.)
If $A$ and $B$ are two sets in $G$, we write $A+B$ for the set of all sums $a+b$, $a \in A$, $b \in B$. Similarly, we write $A-B$ for the set of all differences $a-b$, $a \in A$, $b \in B$. We denote by $\one_E$ the indicator function for the set $E \subseteq G$. 
\begin{definition}[Packing and tiling]
A nonnegative measurable function $f: G \mapsto \RR$ is said to \emph{pack} $G$
with the set (of translates) $T \subseteq G$ at level $L \ge 0$ if 
\begin{equation*}
\sum_{t \in T} f(x-t) \leq L \quad \text{ for a.e. } x \in G . 
\end{equation*}
We then write ``$f+T$ is packing in $G$ at level $L$'', and if $L$ is omitted we understand it to be equal to 1.

A nonnegative function $f: G \mapsto \RR$ \emph{tiles} $G$ at \emph{level} $L$ with the set $T \subseteq G$ if 
\begin{equation*}
\sum_{t \in T} f(x-t) = L \quad \text{ for a.e. } x \in G . 
\end{equation*}
We write ``$f+T$ tiles $G$ at level $L$'' (and if omitted we understand $L=1$). 
The set $T$ is called {\em a tiling complement} of the {\em tile} $f$.

If $f= \one_E$ for some measurable set $E$, then we write ``$E + T$ is a packing'' (or tiling)
rather than ``$\one_E + T$ is a packing'' (or tiling). 
\end{definition}

Denote by $\widehat{G}$ the dual group of continuous characters on $G$.
\begin{definition}[Spectral sets]
A set $\Lambda \subseteq \widehat{G}$ is called a \emph{spectrum} of a measurable
set $E \subseteq G$ if the characters $\{ \lambda \}_{\lambda \in \Lambda}$
form an orthonormal basis in $L^2(E)$. The set $E$ is then called
a \emph{spectral set} of $G$. We say that $E,\Lambda$ are a \emph{spectral pair}.
\end{definition}
Fuglede's conjecture, also known as the {\em spectral set conjecture},
suggests that there is a connection between tilings and spectral sets.
\begin{conjecture}[Fuglede \cite{fuglede}]
A set $E \subset G$ is spectral if and only if it tiles $G$ with some set of translates.
\end{conjecture}
Fuglede's conjecture has motivated research on spectral sets for decades.
It is now known to be false in both directions when $G= \RR^d$, for $d \geq 3$ (see \cite{T, M, KM1, KM2, FMM, FR}),
but the conjecture remains open in several interesting groups.
Certain positive results also exist.
For instance, the conjecture is true for unions of two intervals
in $\RR$ \cite{laba}, and for convex domains in $\RR^2$ \cite{iosevich}.
Recently it was also established that the conjecture holds in
$G = \ZZ_p \times \ZZ_p$ for any prime $p$ \cite{mayeli}.

We will focus on the case when $G$ is a finite abelian group;
that is a finite direct product of finite cyclic groups.
Recall that every finite cyclic group of order $N$ is isomorphic to $\ZZ_N = \ZZ / (N\ZZ)$,
the additive group of residues $\modu N$.
The dual group $\widehat{\ZZ_N}$ of $\ZZ_N$ is the collection of characters $\{ e_n \}$, where 
\begin{equation*}
e_n(x) = \exp 2\pi i n x / N ,
\end{equation*}
for $n=0,\dots, N-1$.
We thus identify $\widehat{\ZZ_N}$ with $\ZZ_N$ in the natural way.
For a function $f : \ZZ_N \mapsto \CC$, we define its Fourier transform $\widehat{f}$ as
\begin{equation*}
\widehat{f} (x) = \sum_{k=0}^{N-1} f(k) e^{-2 \pi i x k/N} .
\end{equation*}

Now suppose that $\Lambda \subseteq \widehat{\ZZ_N} \simeq \ZZ_N$ is a spectrum of $E \subseteq \ZZ_N$.
In finite groups, the spectral relation is symmetric, so, equivalently, $E$ is a spectrum of $\Lambda$.
It is not difficult to show (see, for instance, \cite{KM2}) that the orthogonality of the set of exponentials
$\{ e_{\lambda} \, : \, \lambda \in \Lambda \}$ is equivalent to the condition
\begin{equation*}
\sum_{\lambda \in \Lambda} \left| \widehat{\one_E} \right|^2 (x-\lambda) \leq |E|^2, \quad \text{ for all } x \in  \ZZ_N,
\end{equation*}
where $|E|$ denotes the size of $E$. Moreover, the orthogonality is also equivalent to the condition
\beql{eq:intersec}
\Lambda - \Lambda  \subseteq \{ 0 \} \cup \left\{ \widehat{\one_E} = 0 \right\} .
\eeq
The orthogonality and completeness of the set $\{ e_{\lambda} \, : \, \lambda \in \Lambda \}$
is equivalent to the tiling condition
\beql{eq:speccond}
\sum_{\lambda \in \Lambda} \left| \widehat{\one_E} \right|^2 (x-\lambda) = |E|^2, \quad \text{ for all } x \in  \ZZ_N. 
\eeq
In other words, $\Lambda$ is a spectrum of $E$ if and only if
$|\widehat{\one_E}|^2+ \Lambda$ is a tiling of $\ZZ_N$ at level $|E|^2$.

It is obvious that in $\RR^d$, every cube $Q$ is both a spectral set and a tiling set (a spectrum of $[0,1]^d$, for instance,
is $\ZZ^d$).
Hence, Fuglede's conjecture is trivially true in this special case.
Moreover, the spectra of cubes in $\RR^d$ have been characterized, at least to the extent
that tiling complements of the cube are known.
\begin{othm}{\rm (\cite{JP, IP, K00, LRW})}\label{thm:cubes}
Let $\Lambda$ be a subset of $\RR^d$.
Then $\Lambda$ is a spectrum for the unit cube $Q=[0,1]^d$ if and only if $Q + \Lambda$ tiles $\RR^d$ at level 1.
\end{othm}
We remark that when a domain scales then its tiling complements scale in the same way while its spectra scale reciprocally.
Thus, a corollary of Theorem \ref{thm:cubes} is that the spectra of the rectangle
$$
R = [0, a_1] \times \cdots \times [0, a_d] \subseteq \RR^d
$$
are precisely the tiling complements of the ``dual'' rectangle
$$
R^* = \left[0, \frac{1}{a_1}\right] \times \cdots \times \left[0, \frac{1}{a_d}\right],
$$ 
and one can also make a more general statement about the spectra of linear images of the cube (parallelepipeds).

In this paper we consider the analogous problem of characterizing the spectra of discrete cubes in products of finite cyclic groups.
Let $A_1, \ldots , A_N$ be positive integers, and write
$$
G = \ZZ_{A_1} \times \cdots \times \ZZ_{A_N},
$$
from which we also obtain the isomorphism
$$
\ft{G} \simeq G = \ZZ_{A_1} \times \cdots \times \ZZ_{A_N}.
$$
If $a \geq 1$ is an integer, we write
\begin{equation*}
[a] = \{ 0, 1, 2, \ldots , a-1 \},
\end{equation*} 
and we define the \emph{cube} (in $G$)
\begin{equation*}
Q_{a_1, \ldots , a_N} = [a_1] \times [a_2] \times \cdots \times [a_N] , 
\end{equation*}
as well as its \emph{dual} cube (in $\ft{G}$)
\begin{equation*}
Q_{a_1, \ldots , a_N}^* = Q_{A_1/a_1 , \ldots , A_N/a_N}
\end{equation*}
whenever $a_1, \ldots , a_N$ divide $A_1, \ldots A_N$, respectively. Our main result is a characterization of the spectra of such discrete cubes, analogous to the one valid for cubes in $\RR^d$.
\begin{theorem}
\label{thm:disccubes}
Consider the cube $Q_{a_1, \ldots a_N}$ in $G = \ZZ_{A_1} \times \cdots \times \ZZ_{A_N}$. The condition
\beql{divisibility}
a_1 \mid A_1, \ldots , a_N \mid A_N
\eeq
is necessary and sufficient for $Q_{a_1, \ldots a_N}$ to be a tile and also for it to be spectral.

Suppose that \eqref{divisibility} holds and let  $\Lambda \subseteq G$.
Then $\Lambda$ is a tiling complement of the cube $Q_{a_1, \ldots a_N}$
if and only if $\Lambda$ is a spectrum of the dual cube $Q_{a_1, \ldots a_N}^*$.
\end{theorem}
We see that whereas any cube in $\RR^d$ both tiles and has a spectrum, this is not the case for discrete cubes in $G$, where both properties rest on the condition $a_1 \mid A_1, \ldots , a_N \mid A_N$. Accordingly, Fuglede's conjecture holds for discrete cubes in $G$. This is not difficult to show. The main content of Theorem \ref{thm:disccubes} is the identification of tiling complements of the dual cube with the spectra of the cube.

\begin{observation}\label{obs:subgroup}
Suppose $E \subseteq H \subseteq G$, where $H$ is a subgroup of the finite group $G$. Then
$$
E \text{ tiles } G \Longleftrightarrow E \text{ tiles } H,
$$
and
$$
E \text{ is spectral in } G \Longleftrightarrow E \text{ is spectral in } H.
$$
\end{observation}

Indeed if $E$ tiles $G$ then its translates are completely contained in cosets of $H$, therefore $H$ is tiled itself by
copies of $E$. Conversely, if $E$ tiles $H$ then one only has to copy this tiling in every coset of $H$ in order to obtain
a tiling of $G$.

To see the corresponding equivalence for spectrality assume that $E$ is spectral in $G$. Since any character of $G$
is also a character of $H$, when restricted to $H$, it follows that $E$ is spectral in $H$. And if $E$ is spectral in $H$
then it is also spectral in $G$ as every character of $H$ can be extended to a character of $G$.

Because of Observation \ref{obs:subgroup}, when studying the tiling or spectral properties of $E \subseteq G$
we may always view $E$ as a subset of the group it generates, $\gen{E}$, and decide the question in this setting.
We obtain thus Corollary \ref{cor:spaced} below for ``dilations'' of the cubes,
thus establishing the Fuglede Conjecture for the more general class of sets of type \eqref{spaced}.

\begin{corollary}\label{cor:spaced}
Suppose
\beql{spaced}
E = s_1[k_1] \times s_2[k_2] \times \cdots \times s_N[k_N] \subseteq \ZZ_{A_1} \times \cdots \times \ZZ_{A_N},
\eeq
where
$$
s[k] = s\Set{0, 1, \ldots, k-1} = \Set{0, s, 2s, \ldots, (k-1)s},
$$
and we are assuming that all points in $s_j[k_j]$ are distinct $\bmod\ A_j$, $j=1,2,\ldots,N$.
Write $A_j = A_j' (A_j, s_j)$ and $s_j = s_j' (A_j, s_j)$.

Then $E$ is spectral if and only if it is a tile,
and this happens exactly when
$$
k_j \mid A_j',\ \ \ j=1,2,\ldots,N.
$$
Furthermore, the set
$$
\Lambda \subseteq \ZZ_{A_1} \times \cdots \times \ZZ_{A_N}
$$
is a spectrum for $E$ if and only if the set
$$
\tilde{\Lambda} = \Set{ (s_1' \lambda_1 \bmod A_1', \ldots, s_N' \lambda_N \bmod A_N'):\ (\lambda_1,\ldots,\lambda_N) \in \Lambda}
$$
is a tiling complement of the cube
$$
\tilde Q = [A_1'/k_1]\times\cdots\times[A_N'/k_N]
$$
in the group $\ZZ_{A_1'} \times\cdots\times \ZZ_{A_N'}$.

\end{corollary}

The proofs of Theorem \ref{thm:disccubes} and Corollary \ref{cor:spaced} are given in \S\ref{sec:proofs}.

\section{Proofs}\label{sec:proofs}
The proof of Theorem \ref{thm:disccubes} is essentially the same
regardless of the number $N$ of finite group factors in the product group $G$.
We therefore prove Theorem \ref{thm:disccubes} in the special case
when $G= \ZZ_A \times \ZZ_B$ and $Q_{a,b} = [a] \times [b]$.

We will need the following lemma.
\begin{lemma}\label{lm:fourier-zeros}
Let $f$ be the indicator function of $Q_{a,b} \subseteq G=\ZZ_A \times \ZZ_B$.
Then if $Z(\ft{f} )$ is the set of zeros of the Fourier Transform of $f$ in $\ft{G} \simeq G$ we have
\beql{eq:zeros}
Z(\ft{f}) = \Set{(j,k) \neq (0,0):\ \frac{A}{(A,a)}\mid j \, \mbox{ or } \, \frac{B}{(B,b)}\mid k}.
\eeq
Note also that $Z(\ft{f})$ does not intersect the difference set 
\beql{diff-set}
Q_{\frac{A}{(A,a)}, \frac{B}{(B,b)} } - Q_{\frac{A}{(A,a)}, \frac{B}{(B,b)}}.
\eeq
\end{lemma}
\begin{proof}
We have that 
\begin{equation*}
\ft{f}(j,k) = \ft{\one_{[a]}}(j) \cdot \ft{\one_{[b]}}(k) , 
\end{equation*}
where the indicator functions $\one_{[a]}$ and $\one_{[b]}$ are defined on the groups $\ZZ_A$ and $\ZZ_B$, respectively. Hence, $\ft{f}$ vanishes if and only if either $\ft{\one_{[a]}}$ or $\ft{\one_{[b]}}$ is zero. This gives the conditions in \eqref{eq:zeros}. 

The set in \eqref{diff-set} is the cube
\beql{2cube}
\Set{-\left(\frac{A}{(A,a)}-1\right),\ldots,\frac{A}{(A,a)}-1} \times
\Set{-\left(\frac{B}{(B,b)}-1\right),\ldots,\frac{B}{(B,b)}-1} ,
\eeq
which clearly does not intersect $Z(\ft{f})$.
\end{proof}

\begin{proof}[Proof of Theorem \ref{thm:disccubes}]
Notice first that \eqref{divisibility} is obviously
necessary and sufficient for $Q_{a,b}$ to be a tile.
Moreover, it is clear that \eqref{divisibility} is sufficient for $Q_{a,b}$ to be spectral, as 
$$
\Set{(x,y) \in \ZZ_A \times \ZZ_B \, : \, \frac{A}{a} \mid x , \, \frac{B}{b} \mid y}
$$
is then one possible spectrum of $Q_{a,b}$. We will see below that \eqref{divisibility} is also a necessary condition for spectrality.

Suppose now that $Q_{a,b}$ has $\Lambda$ as a spectrum.
Write $f$ for the indicator function of $Q_{a,b}$ and observe that
$\Lambda-\Lambda \setminus\Set{0}$ does not intersect the difference set of
$$
Q_{\frac{A}{(A,a)}, \frac{B}{(B,b)}}
$$
according to \eqref{eq:intersec} and Lemma \ref{lm:fourier-zeros}.
Hence $Q_{\frac{A}{(A,a)}, \frac{B}{(B,b)}} + \Lambda$ is a packing in $G$, so that
$$
\Abs{Q_{\frac{A}{(A,a)}, \frac{B}{(B,b)}}} \cdot \Abs{\Lambda} \le \Abs{G}.
$$
Since $\Lambda$ is a spectrum of $Q_{a,b}$ it follows that $\Abs{\Lambda} = \Abs{Q_{a,b}}$, so the above
inequality reads
\beql{packing}
\frac{A}{(A,a)} \frac{B}{(B,b)} a b \le AB.
\eeq
The only way this can happen is if it is an equality (as $a/(A,a) \ge 1, b/(B,b) \ge 1$)
and this implies $a \mid A$ and $b \mid B$. The dual cube is defined in this case,
and since the inequality in \eqref{packing} is actually an equality it follows
that the packing $Q_{\frac{A}{(A,a)}, \frac{B}{(B,b)}} + \Lambda$ is in fact a tiling of $G$, as we had to show.
We have shown that if $\Lambda$ is a spectrum of $Q_{a,b}$ then \eqref{divisibility} holds and $\Lambda$ is a tiling
complement of the dual cube.

For the converse suppose that $a \mid A$ and $b \mid B$,
so that the dual cube $Q_{a, b}^*$ of $Q_{a, b}$ exists,
and suppose also that $Q_{a, b}^* + \Lambda$ is a tiling of $\ft{G} \simeq G$.
Taking Fourier Transforms on the tiling condition
$$
\one_{Q_{a, b}^*} * \one_\Lambda = 1,
$$
we get that
$$
\ft{\one_{Q_{a, b}^*}} \cdot \ft{\one_\Lambda} = A B \, \one_{\Set{0}} ,
$$
which implies that $\ft{\one_\Lambda}$ is supported on the set $\Set{\ft{\one_{Q_{a, b}^*}} = 0} \cup \Set{0}$,
and, according to Lemma \ref{lm:fourier-zeros}, this latter set is contained in the complement of
\beql{2c}
\Set{-(a-1),\ldots,a-1}\times \Set{-(b-1),\ldots,(b-1)}.
\eeq
Thus $\ft{\one_\Lambda}$ is supported at $0$ plus the complement of the support of $\one_{Q_{a,b}} * \one_{-Q_{a,b}}$. We have that 
$$
\ft{\one_{\Lambda}} \cdot \left( \one_{Q_{a,b}} * \one_{-Q_{a,b}} \right) = |Q_{a,b}|^2 \one_{\{ 0 \}} ,
$$
and by taking the inverse Fourier Transform we get
$$
\one_{\Lambda} * \left| \ft{\one_{Q_{a,b}}} \right|^2 = |Q_{a,b}|^2 .
$$
Hence $|\ft{\one_{Q_{a,b}}}|^2 + \Lambda$ tiles $G$ at level $|Q_{a,b}|^2$, and by \eqref{eq:speccond} this is precisely what it means for $\Lambda$ to be a spectrum of $Q_{a,b}$.
\end{proof}

\begin{proof}[Proof of Corollary \ref{cor:spaced}]
If $s[k] \subseteq \ZZ_A$ then
\begin{align}
\gen{s[k]} &= \gen{s} \nonumber\\
  &= \Set{n s \bmod A: n\in\ZZ} \label{subgroup}\\
  &= \Set{0 \bmod A, s \bmod A, 2s \bmod A, \ldots, \left(\frac{A}{(A,s)}-1\right) s \bmod A} \nonumber\\
  &\simeq \ZZ_{A'},\nonumber
\end{align}
where $A'=A/(A,s)$.
It follows that
\beql{tmp-iso}
\gen{E} \simeq \ZZ_{A_1'} \times \cdots \times \ZZ_{A_N'},
\eeq
and, under the obvious isomorphism
$$
(0,\ldots,0,s_j,0,\ldots,0) \to (0,\ldots,0,1,0,\ldots,0)
$$
implied in \eqref{tmp-iso}, the image of $E$ is the cube
$$
Q = [k_1]\times\cdots\times [k_N].
$$
So, to decide if $E$ tiles the original group $\ZZ_{A_1}\times\cdots\times \ZZ_{A_N}$ or is spectral therein we can
equivalently answer the same question for $Q$ in the group \eqref{tmp-iso}.
According to Theorem \ref{thm:disccubes} the cube $Q$ tiles the group \eqref{tmp-iso} if and only if
$$
k_j \mid A_j' = \frac{A_j}{(A_j,s_j)},\ \ \text{ for } j=1,2,\ldots,N,
$$
and the same condition is equivalent to $Q$ being spectral in the same group.

If $(\lambda_1,\ldots,\lambda_N) \in \ZZ_{A_1}\times\cdots\times\ZZ_{A_N}$ is a character
on $\ZZ_{A_1}\times\cdots\times\ZZ_{A_N}$ then, restricted on the subgroup $\gen{E}$ viewed as in \eqref{subgroup},
it becomes the character
$$
(s_1'\lambda_1,\ldots,s_N'\lambda_N) \in \ZZ_{A_1'} \times \cdots \times \ZZ_{A_N'}.
$$
Therefore, for the collection of characters $\Lambda$ on the original group $\ZZ_{A_1}\times\cdots\times\ZZ_{A_N}$
to form a spectrum of $E$ it is necessary and sufficient
that the collection $\tilde\Lambda$ of characters on $\ZZ_{A_1'} \times \cdots \times \ZZ_{A_N'}$
form a spectrum of $Q$, and this is equivalent to $\tilde\Lambda$ being a tiling complement
of the dual cube of $Q$ in $\ZZ_{A_1'} \times \cdots \times \ZZ_{A_N'}$,
which is the cube
$$
[A_1'/k_1]\times\cdots\times[A_N'/k_N].
$$

\end{proof}

\end{document}